\documentclass[10pt, a4paper]{article}
\usepackage{amsfonts}
\usepackage{amsbsy}
\usepackage{amssymb}
\usepackage{amsmath}
\usepackage{amsthm}
\newtheorem{proposition}{Proposition}[section]
\newtheorem{theorem}{Theorem}[section]
\newtheorem{corollary}{Corollary}[section]
\newtheorem{lemma}{Lemma}[section]

\theoremstyle{definition}
\newtheorem{definition}{Definition}[section]
\newtheorem{remark}{Remark}[section]

\begin{document}

\author{Cristian Ida and Paul Popescu}
\title{Coeffective basic cohomologies of $K$--contact and Sasakian manifolds}
\date{}
\maketitle
\title{}

\begin{abstract}
In this paper we define coeffective de Rham cohomology for basic forms on a $K$--contact or Sasakian manifold $M$ and we discuss its relation with usually basic cohomology of $M$. When $M$ is of finite type (for instance it is compact) several inequalities relating some basic coeffective numbers to classical basic Betti numbers of $M$ are obtained. In the case of Sasakian manifolds, we define and study coeffective Dolbeault and Bott-Chern cohomologies for basic forms. Also, in this case, we prove some Hodge decomposition theorems for coeffective basic de Rham cohomology, relating this cohomology with coeffective basic Dolbeault or Bott-Chern cohomology. The notions are introduced in a similar manner with the case of symplectic and K\"{a}hler manifolds.
\end{abstract}

\medskip 
\begin{flushleft}
\strut \textbf{2010 Mathematics Subject Classification:} 53C25; 57R18; 58A12; 58A14.

\end{flushleft}

\section{Introduction and preliminary notions}
\setcounter{equation}{0}

\subsection{Introduction}

The coeffective cohomology was introduced by Bouch\'{e} in \cite{Bo} for symplectic manifolds. More exactly, a symplectic form $\omega$ defines a special subcomplex of the de Rham complex $(F^\bullet(M),d)$ of differential forms on $M$: it consists of those
forms $\varphi$ which are annihilated by $\omega$, that is, $\varphi\wedge\omega=0$. Since $\omega$ is closed, we have in fact
a subcomplex of $(F^\bullet(M),d)$ whose cohomology is called coeffective. This cohomology is related with the truncated de Rham cohomology by the class $\omega$. Further signifiant developments of coeffective cohomologies in many different contexts (symplectic, K\"{a}hler, (almost) cosymplectic, (almost) contact, quaternionic manifolds) are given by a series of papers of de Andr\'{e}s, Fern\'{a}ndez,  de Le\'{o}n,  Ib\'{a}\~{n}ez,  Menc\'{i}a, Chinea,   Marrero \cite{A-F-L-I-M, C-L-M, C-L-M1, F-I-L, F-I-L1, F-I-L2} and others papers by these authors. For K\"{a}hler manifolds both cohomologies (coeffective cohomology and de Rham cohomology truncated by $[\omega]$) are isomorphic for $p\neq n$, $\dim M=2n$, though in general they are different for non K\"{a}hler symplectic
manifolds \cite{A-F-L-I-M}. For symplectic manifolds of finite type was introduced the coeffective numbers of the symplectic manifold and several inequalities relating them to the Betti numbers. Similar results were obtained in the context of almost contact \cite{C-L-M1, F-I-L1} and quaternionic manifolds \cite{F-I-L2}. Also, a coeffective Dolbeault cohomology for compact K\"{a}hler and indefinite K\"{a}hler manifolds is studied in \cite{I}.

Our aim in this paper is to study the coeffective de Rham, Dolbeault and Bott-Chern cohomologies for basic forms of (compact) $K$--contact or Sasakian manifolds with respect to the Reeb foliation $\mathcal{F}_{\xi}$ of the fundamental Reeb vector field $\xi$, giving a new contribution concerning to basic cohomology of $K$--contact or Sasakian manifolds. Notice that a background about the basic cohomology of $K$--contact and Sasakian manifolds can be found in the Ch. VII from \cite{B-G}. Other developments of basic cohomologies of Sasakian manifolds in a similar direction as in the recent studies for symplectic manifolds, \cite{T-Y1}, are given in \cite{I-P}. 

The structure of paper is the following: 

In the preliminary subsection, following \cite{B-G, Pi} we briefly recall some elementary definitions about basic forms, basic star operator, basic de Rham Laplacian, basic de Rham cohomology on $K$--contact manifolds. 

We notice that if $\eta$ is the contact form of $M$ then $d\eta$ is basic with respect to the Reeb foliation $\mathcal{F}_{\xi}$. Thus, in the Section 2 we begin our study with the coeffective de Rham cohomology for basic forms on a $K$--contact manifold $M$. The main ingredient is given by the isomorphism between the space of basic  differential forms on $M$ and the space of differential forms on the orbit space $M_\xi$ of  $\mathcal{F}_\xi$, which is known that it is symplectic (or K\"{a}hlerian in the Sasakian case). Thus, following the classical study of coeffective cohomology of symplectic manifolds, see \cite{F-I-L}, we define the coeffective basic de Rham cohomology of $M$ and we prove that when $M$ is compact Sasakian 
\begin{displaymath}
H^p(\mathcal{A}_b(M))\cong\widetilde{H}_b^p(M)\,,\,\,\forall\,p\neq n,
\end{displaymath}
where $H^p(\mathcal{A}_b(M))$ denotes the coeffective basic cohomology group of degree $p$ of $M$ and $\widetilde{H}_b^p(M)$ is the subspace of the basic de Rham cohomology group of $M$ consisting of those classes $a\in H_b^p(M)$ such that $a\wedge[d\eta]=0$, or in other words, the truncated basic de Rham cohomology group of degree $p$. Notice that for an arbitrary $K$--contact manifold $H^p(\mathcal{A}_b(M))$ vanishes for every $p\leq n-1$, where $\dim M=2n+1$. Also, using a technique based on the long exact sequence in cohomology associated with an exact short sequence of complexes, we obtain that the coeffective basic de Rham cohomology groups of a $K$--contact or Sasakian manifold manifold $M$ of finite type have finite dimension. In this case, if we denote by $c_b^p(M)=\dim H^p(\mathcal{A}_b(M))$, called the coeffective basic numbers of order $p$ of $M$, then  they satisfy the following inequalities:
\begin{displaymath}
b_b^p(M)-b_b^{p+2}(M)\leq c_b^p(M)\leq b_b^p(M)+b_b^{p+1}(M),\,\,\forall\, p\geq n+1,
\end{displaymath}
where $b^p_b(M)=\dim H^p_b(M)$ is the  basic Betti number of order $p$ of $M$.

As a consequence, for a compact Sasakian manifolds, we deduce that
\begin{displaymath}
c_b^p(M)=b_b^p(M)-b_b^{p+2}(M),\,\,\forall \,p\geq n+1,
\end{displaymath}
which means that the coeffective basic numbers of a compact Sasakian manifold measure the jumps between the basic Betti numbers.

In the end of Section 2, using an exact sequence in cohomology (which is the foliation analogue of the Gysin sequence), we prove that the following isomorphism holds:
\begin{displaymath}
H^p(\mathcal{A}(M))\cong H^{p-1}(\mathcal{A}_b(M)),\,\,\forall\, p=1,\ldots,2n+1,
\end{displaymath}
where $H^\bullet(\mathcal{A}(M))$ is the coeffective de Rham cohomology of $M$ considered in \cite{F-I-L1}.

In Section 3 we consider basic forms with complex valued on a Sasakian manifold $M$ and taking into account that $d\eta$ is a basic form of complex type $(1,1)$, in a similar manner with coeffective Dolbeault cohomology of K\"{a}hler manifolds, see \cite{I}, we define and study coeffective basic Dolbeault cohomology of a Sasakian manifold. We prove that when $M$ is compact 
\begin{displaymath}
H^{r,s}(\mathcal{A}_b(M))\cong\widetilde{H}_b^{r,s}(M)\,,\,\,\forall\,r+s\neq n,
\end{displaymath}
where $H^{r,s}(\mathcal{A}_b(M))$ denotes the coeffective basic cohomology group of type $(r,s)$ of $M$ and $\widetilde{H}_b^{r,s}(M)$ is the  truncated basic Dolbeault cohomology group of type $(r,s)$. In the case when
$M$ is a compact Sasakian manifold, we prove a Hodge decomposition theorem for coeffective basic de Rham cohomology of $M$, relating this cohomology with coeffective basic Dolbeault cohomology of $M$. Also, several inequalities relating the coeffective basic Hodge numbers to the classical basic Hodge numbers are given similarly as for in the de Rham case.

The aim of Section 4 is to construct a coeffective basic Bott-Chern cohomology of a Sasakian manifold $M$. In this sense we firstly define basic Bott-Chern and Aeppli cohomology of $M$ and we obtain a Hodge-Bott-Chern decomposition theorem for basic forms of $M$. Next, in similar manner with the study of coeffective basic de Rham and Dolbeault cohomology of $M$, we define and study a coeffective Bott-Chern cohomology for basic forms on $M$.

The main methods used here are similarly and closely related to those used in \cite{Bo, C-L-M1, F-I-L, F-I-L1, F-I-L2}.

\subsection{Preliminaries}

Let $(M,F,\xi,\eta,g)$ be a $(2n+1)$--dimensional \textit{almost contact manifold}; that is see \cite{Bl1, Bl, B-G, Pi}, $F$ is an $(1,1)$--tensor field, $\eta$ is an $1$--form, $\xi$ is a vector field, and $g$ is a Riemannian metric on $M$ such that
\begin{equation}
\label{defcontact}
F^2=-{\rm Id}+\eta\otimes\xi,\,\eta(\xi)=1,\,\,\,{\rm and}\,\,\,g(FX,FY)=g(X,Y)-\eta(X)\eta(Y)
\end{equation}
for every $X,Y\in\mathcal{X}(M)$, where ${\rm Id}$ is the identity transformation. Then we have $F(\xi)=0$ and $\eta(X)=g(X,\xi)$ for all $X\in\mathcal{X}(M)$. The \textit{fundamental $2$--form} $\Phi$ of $M$ is defined by $\Phi(X,Y)=g(X,FY)$, and the $(2n+1)$--form $\eta\wedge\Phi^n$ is a volume form on $M$. The almost contact metric manifold is said to be: \textit{contact} if $d\eta=\Phi$; \textit{$K$--contact} if it is contact and $\xi$ is Killing; \textit{normal} if $[F,F]+2d\eta\otimes\xi=0$; \textit{Sasakian} if it is contact and normal. If $M$ is Sasakian manifold then it is $K$--contact \cite{Bl1}.

Consider the field $F^0(M)=\mathcal F(M)$ of smooth real valued functions defined on $M$. For each $p =1,\ldots,2n+1$ denote by $F^{p}(M)$ the module of $p$--forms, by $F(M)=\oplus_{p \geq 0}F^{p}(M)$ the exterior algebra of $M$,  and by $\left\langle, \right\rangle$, the natural scalar product on $F(M)$. 

Recall that the differential form $\omega$ on $M$ is called \textit{basic} if it is \textit{horizontal} (that is $\imath_{\xi}\omega=0$, where $\imath_{\xi}$ denotes the interior product with respect to $\xi$) and \textit{invariant} (that is  $\mathcal L_{\xi}\omega=0$, where $\mathcal L_{\xi}$ denotes the Lie derivative with respect to $\xi$). Denote by $F^{p}_{b}(M)$ the subspace of all basic $p$--forms on the manifold $M$. It is a module over the ring $F^{0}_{b}(M)=\mathcal F_{b}(M)$ of \textit{basic functions} on $M$ (that is, $\mathcal{L}_{\xi}f=0$) and let $F_b(M)=\oplus_{p \geq 0} F^{p}_{b}(M)$ be the graded algebra of all basic forms on $M$. By Cartan identity $\mathcal L_{\xi}=d\imath_{\xi}+\imath_{\xi}d$, we easily obtain that the exterior differential of a basic form is also basic, so we can consider the basic differential $d_{b}=d|_{F^{p}_{b}(M)}:F^{p}_{b}(M)\longrightarrow F^{p+1}_{b}(M)$. 

Thus, the basic forms constitute a subcomplex 
$\left(\oplus_{p \geq 0} F^{p}_{b}(M),d_{b}\right)$ 
of the de Rham complex $\left(\oplus_{p \geq 0} F^{p}(M),d\right)$. The cohomology of this subcomplex is defined by
\[
H_{b}(M)=\oplus_{p\geq 0}H^{p}_{b}(M), \quad H^{p}_{b}(M)=\ker \{d_{b}:F^p_b(M)\rightarrow F^{p+1}_b(M)\}/ d_b(F_b^{p-1}(M)).
\]
This cohomology play the role of de Rham cohomology of the orbit space of the $K$--contact manifold $M$ and we call it the \textit{basic de Rham cohomology} or simply the \textit{basic cohomology} of $M$. Moreover, the space of basic cohomology $H^{\bullet}_{b}(M)$ is an invariant of the characteristic foliation $\mathcal F_{\xi}$ and therefore is an invariant of the $K$--contact structure on the manifold $M$. The relation between the basic cohomology $H^{\bullet}_{b}(M)$ and de Rham cohomology $H^{\bullet}(M)$ of the $K$--contact manifold $M$ is the same as in the general case of a foliation generated by a nonsingular Killing vector field (see for instance \cite{To88}, Theorem 10.13, pg. 139). On compact $K$--contact manifolds the basic cohomology groups enjoy some special properties. In particular, there is a transverse Hodge theory \cite{B-G, E-K1, E-K-H}. 

Let $\star$ be the usual star operator on $M$. If $\omega \in F^{p}_{b}(M)$ then 
the $(2n-p)$-form $\imath_{\xi}\star \omega$ is basic. Therefore we can define the \textit{basic star operator} $\star_{b}:F^{p}_{b}(M)\longrightarrow F^{2n-p}_{b}(M)$ by
\begin{eqnarray}\label{e376}
\star_{b}\omega=(-1)^{p}\imath _{\xi} \star \omega.
\end{eqnarray}
Also, the usual scalar product $\left\langle , \right\rangle$ on $F^p(M)$ restricted to basic forms, is denoted by $\left\langle , \right\rangle_b$, and it is given by
\begin{eqnarray}\label{sca}
\left\langle \omega,\theta\right\rangle_{b}=\int_{M}\omega \wedge \star_{b}\theta \wedge \eta,
\end{eqnarray}
for all $\omega,\theta \in F^{p}_{b}(M)$ and we denote by the symbol $A^\ast$ the adjoint of the operator $A: F_b(M) \rightarrow F_b(M)$ with respect to $\left\langle , \right\rangle_b$. As it is well known, the $\left\langle , \right\rangle_b$--adjoint $d^\ast_b$ of $d_b$ satisfies $d^*_{b}=-\star_{b} d_{b} \star_{b}$.

The \textit{basic de Rham Laplacian} $\Delta_b$ is defined in terms of $d_b$ and its adjoint $d_b^*$ by
\begin{equation}
\label{baslap1}
\Delta_b=d_bd_b^\ast+d_b^\ast d_b.
\end{equation}
The space $\mathcal{H}_b^p(M)$ of \textit{basic harmonic $p$--forms} on $M$ is then defined to be the kernel of $\Delta_b:F_b^p(M)\rightarrow F_b^p(M)$, and $\mathcal{H}_b^p(M)=\ker d_b\cap \ker d_b^*$. The transverse Hodge theorem \cite{E-K-H} then says that 
\begin{equation}
\label{hodgedec}
F_b^p(M)={\rm Im\,}d_b\oplus{\rm Im\,}d_b^*\oplus\ker\Delta_b
\end{equation}
see also \cite{CML97}, and each basic cohomology class has an unique harmonic representative, i.e. 
\begin{equation}
\label{harmoniciso}
H^p_b(M)\cong \mathcal{H}_b^p(M).
\end{equation}

\section{Coeffective de Rham cohomology for basic forms}
\setcounter{equation}{0}

Throughout this section $M$ is a (compact) $K$--contact manifold of dimension $2n+1$ and sometimes $M$ is Sasakian. We start with a fundamental result which play an important role for our purpose.
\begin{theorem}
\label{sym}
(\cite{Pi}) Let $M$ be a  $K$--contact manifold of dimension $2n+1$ and $M_{\xi}$ the orbit space of the Reeb foliation $\mathcal{F}_{\xi}$ defined by $\xi$. If $\pi:M\rightarrow M_{\xi}$ is the natural projection then $\pi^*:F^p\left(M_{\xi}\right)\rightarrow F_b^p(M)$ is an isomorphism.
\end{theorem}
\begin{proof}
Obviously, $\pi^*$ is injective. 

We prove now that for any $\varphi\in F_b^p(M)$ there exists $\varphi^{\prime}\in F^p(M_{\xi})$ such that $\varphi=\pi^*\varphi^{\prime}$. Since $\varphi$ is horizontal (that is $\imath_{\xi}\varphi=0$), the values $\varphi(X_1,\ldots,X_p)$  can be non zero only when the tangent vectors $\{X_1,\ldots,X_p\}\in T_xM$ are orthogonal on $\xi$. But the condition $\mathcal{L}_{\xi}\varphi=0$ shows that $\varphi$ is invariant by the Reeb group $\{\Phi_t\}_{t\in\mathbb{R}}$, that is $\Phi_t^*\varphi=\varphi$. It follows that
\begin{displaymath}
\varphi(\Phi_{t*}X_1,\ldots,\Phi_{t*}X_p)=\varphi(X_1,\ldots,X_p)
\end{displaymath}
and so at the point $\pi(x)$ is welll-defined a $p$--form $\varphi^{\prime}_{\pi(x)}$ with the property $\varphi_x=\pi^*\varphi^{\prime}_{\pi(x)}$. But $x$ is arbitrary in $M$ and then is well-defined the $p$--form $\varphi^{\prime}\in F^p(M_{\xi})$ with the property $\varphi=\pi^*\varphi^{\prime}$, which proves that $F_b^p(M)\subseteq{\rm Im}\,\pi^*$.

It remains only to prove that ${\rm Im}\,\pi^*\subseteq F_b^p(M)$. Remark that for any $\varphi^{\prime}\in F^p(M_{\xi})$ we have $\Phi^*_t\pi^*\varphi^{\prime}=\pi^*\varphi^{\prime}$, $\imath_{\xi}\pi^*\varphi^{\prime}=0$, hence $\mathcal{L}_{\xi}\pi^*\varphi^{\prime}=0$ and then $\pi^*\varphi^{\prime}\in F_b^p(M)$.
\end{proof}
Also, it is well know that $(M_{\xi},\Omega)$ is symplectic with  $d\eta=\pi^*\Omega$, and when $M$ is Sasakian then $(M_{\xi},\Omega)$ is K\"{a}hlerian. We have now
\begin{lemma}
\label{surj}
The operator $L:F_b^{p}(M)\rightarrow F_b^{p+2}(M)$ defined by $L\varphi=\varphi\wedge d\eta$ is injective for $p\leq n-1$ and surjective for $p\geq n-1$.
\end{lemma}
\begin{proof}
According to \cite{Bo}, we have that the symplectic operator $L^{\prime}:F^{p}(M_{\xi})\rightarrow F^{p+2}(M_{\xi})$ given by $L^{\prime}\varphi^{\prime}=\varphi^{\prime}\wedge\Omega$ is surjective for $p\geq n-1$ and injective for $p\leq n-1$. Now, by Theorem \ref{sym}, for every $\varphi\in F_b^p(M)$ there is $\varphi^{\prime}\in F^{p}(M_{\xi})$  such that $\varphi=\pi^*\varphi^{\prime}$ and 
\begin{displaymath}
L\varphi=\varphi\wedge d\eta=\pi^*\varphi^{\prime}\wedge\pi^*\Omega=\pi^*(\varphi^{\prime}\wedge\Omega)=(\pi^*\circ L^{\prime})\varphi^{\prime}.
\end{displaymath}
Thus, the operator $L$ is also injective for $p\leq n-1$ and surjective for $p\geq n-1$.
\end{proof}

Now, as in the case of classical coeffective cohomology, see \cite{Bo, C-L-M1, F-I-L1}, we consider the subspace $\mathcal{A}_b^{p}(M)\subset F_b^{p}(M)$ defined by
\begin{displaymath}
\mathcal{A}_b^{p}(M)=\{\varphi\in F_b^{p}(M)\,|\,\varphi\wedge d\eta=0\}=\ker L|_{F_b^{p}(M)}.
\end{displaymath}
A basic form  $\varphi\in\mathcal{A}_b^{p}(M)$ is said to be a \textit{coeffective basic $p$--form} on $M$.

Since $d_b$ commutes with $L$, we can consider the subcomplex of basic de Rham complex of $M$, namely $\left(\mathcal{A}_b^{\bullet}(M), d_b\right)$, called the \textit{coeffective basic  de Rham complex} of $M$. The cohomology groups of this complex are called \textit{coeffective basic  de Rham cohomology groups} of $M$ and they are denoted by $H^{p}(\mathcal{A}_b(M))$. 

As a consequence of Lemma \ref{surj}, one gets
\begin{proposition}
Let $M$ be a $K$--contact manifold of dimension $2n+1$. Then $\mathcal{A}_b^{p}(M)=\{0\}$ for $p\leq n-1$, therefore
\begin{equation}
\label{R2}
H^{p}(\mathcal{A}_b(M))=\{0\}\,,\,\,{\rm for}\,\,p\leq n-1.
\end{equation}
\end{proposition}

Let us consider the subspace of $\widetilde{H}^{p}_{b}(M)$ given by the basic de Rham cohomology classes truncated by the basic de Rham class $[d\eta]$, namely,
\begin{equation}
\label{R3}
\widetilde{H}^{p}_{b}(M)=\{a\in H^{p}_{b}(M)\,|\,a\wedge[d\eta]=0\}.
\end{equation}
We notice that as in the case of compact cosymplectic manifolds or compact K\"{a}hler manifolds, see \cite{C-L-M1, I}, we can obtain a relation between the coeffective basic de Rham cohomology of a compact Sasakian manifold $M$ and the basic de Rham cohomology of $M$ truncated by $[d\eta]$ in the following way.

Let us  denote by $[\cdot]$ the basic de Rham cohomology classes and by $\{\cdot\}$ the coeffective basic de Rham classes.
\begin{proposition}
\label{alphadR}
For any  $K$--contact manifold $M$ of dimension $2n+1$ the natural mapping
\begin{equation}
\label{R4}
\alpha_{p}\left(\{\varphi\}\right)=[\varphi],
\end{equation}
is surjective for $p\geq n$. 
\end{proposition}
\begin{proof}
Let $a\in\widetilde{H}^{p}_{b}(M)$, that is, $a\in H^{p}_{b}(M)$ and $a\wedge[d\eta]=0$ in $H^{p+2}_{b}(M)$. Consider a representative $\varphi$ of $a$ and suppose that $\varphi\notin\mathcal{A}_b^{p}(M)$ (notice that if $\varphi\in\mathcal{A}_b^{p}(M)$, then $\varphi$ defines a basic cohomology class in $H^{p}(\mathcal{A}_b(M))$ such that $\alpha\left(\{\varphi\}\right)=a$).

Since $a\wedge[d\eta]=0$, then there exists $\sigma\in F_b^{p+1}(M)$ such that $\varphi\wedge d\eta=d_b\sigma$. Then, from Lemma \ref{surj}, there exists $\gamma\in F_b^{p-1}(M)$ such that $L\gamma=\sigma$. Thus, $L(\varphi-d_b\gamma)=0$ and $d_b(\varphi-d_b\gamma)=0$. Hence, $\varphi-d_b\gamma$ defines a basic cohomology class  in $H^{p}(\mathcal{A}_b(M))$ such that $\alpha_{p}(\{\varphi-d_b\gamma\})=a$.
\end{proof}

We also notice that for compact Sasakian manifolds, we have
\begin{equation}
\label{C5}
\Delta_{b}L=L\Delta_b.
\end{equation} 

This follows by a direct calculation using the following known identities from Sasakian geometry \cite{Fu, Og, Pi}, namely
\begin{displaymath}
\Delta\varphi=\Delta_b\varphi+L\Lambda\varphi+e_{\eta}\Lambda d\varphi-e_{\eta}d\Lambda\varphi,
\end{displaymath}
\begin{displaymath}
\Delta L\varphi-L\Delta\varphi=4(n-p-1)L\varphi+4de_{\eta}\varphi,
\end{displaymath}
and
\begin{displaymath}
(\Lambda L^k-L^k\Lambda)\varphi=4k[(n-p-k+1)L^{k-1}\varphi+e_{\eta}\imath_{\xi}L^{k-1}\varphi],
\end{displaymath}
where $L^0\varphi=\varphi$ and $L^{-1}\varphi=0$. Here $\Delta=dd^*+d^*d $ is the usual Laplacian on $M$, $e_{\eta}:F^p(M)\rightarrow F^p(M)$ is defined by $e_{\eta}\varphi=\eta\wedge\varphi$ and $\Lambda=\star L\star=-\star_b L\star_b$ is the adjoint of $L$ with respect to $\langle,\rangle$ and $\langle,\rangle_b$, respectively. 

The relation \eqref{C5} say that the map $L:\mathcal{H}_b^p(M)\rightarrow\mathcal{H}_b^{p+2}(M)$ is well defined on the space of harmonic basic $p$--forms on $M$. Moreover, by Bouch\'{e} result for compact K\"{a}hler manifolds, see \cite{Bo}, and taking into account the Theorem \ref{sym}, we have that
\begin{lemma}
\label{surjharmonic}
Let $M$ be a compact Sasakian manifold of dimension $2n+1$. The operator $L:\mathcal{H}_b^p(M)\rightarrow\mathcal{H}_b^{p+2}(M)$ is surjective for $p\geq n-1$.
\end{lemma}
This, also follows directly using Lemma \ref{surj} and basic Hodge decomposition \eqref{hodgedec}. We have

\begin{theorem}
\label{main1dR}
Let $M$ be a compact Sasakian manifold of dimension $2n+1$. Then
\begin{equation}
\label{X1}
H^p(\mathcal{A}_b(M))\cong \widetilde{H}^p_b(M),\,\,\forall\,p\neq n.
\end{equation}
\end{theorem}
\begin{proof}
The proof it follows by two cases.

Case 1: $p\leq n-1$.

From \eqref{R2} we know that $H^{p}(\mathcal{A}_b(M))=\{0\}$ for $p\leq n-1$. Moreover, from the isomorphism \eqref{harmoniciso}, we have
\begin{equation}
\label{R8}
\widetilde{H}_{b}^{p}(M)\cong\left\{\varphi\in\mathcal{H}_{b}^{p}(M)\,|\,\varphi\wedge d\eta\in d_b\left(F_b^{p+1}(M)\right)\right\}\cong\left\{\varphi\in\mathcal{H}^{p}_{b}(M)\,|\,\varphi\wedge d\eta=0\right\}.
\end{equation}
Thus, from Lemma \ref{surj} we conclude that $\widetilde{H}^{p}_{b}(M)=\{0\}$ for $p\leq n-1$. This finishes the proof for $p\leq n-1$.

Case 2: $p\geq n+1$.

We shall see that the mapping $\alpha_{p}$ given by \eqref{R4} is an isomorphism for $p\geq n+1$. From Proposition \ref{alphadR}, it is suficient to show the injection.

Let $a\in H^{p}(\mathcal{A}_b(M))$ such that $\alpha_{p}(a)=0$ in $\widetilde{H}_{b}^{p}(M)$ and suppose that $\varphi$ is a representative  of $a$. Since $\alpha_{p}(a)=\alpha_{p}\left(\{\varphi\}\right)=[\varphi]=0$ in $\widetilde{H}_{b}^{p}(M)$, there exists $\psi\in F_b^{p-1}(M)$ such that
\begin{displaymath}
\varphi=d_b\psi.
\end{displaymath}
Suppose $\psi\notin\mathcal{A}_b^{p-1}(M)$ (notice that if $\psi\in\mathcal{A}_b^{p-1}(M)$, then $a=0$ and we conclude the proof). Since $L$ commute with $d_b$, then $d_b(L\psi)=L(d_b\psi)=L\varphi=0$; therefore $L\psi$ defines a basic cohomology class $[L\psi]\in H^{p+1}_{b}(M)$. From the isomorphism \eqref{harmoniciso}, we have
\begin{displaymath}
L\psi=h+d_b\gamma,
\end{displaymath}
for $h\in \mathcal{H}^{p+1}_{b}(M)$, $\gamma\in F_b^{p}(M)$. Since $p\geq n+1$ and $h\in\mathcal{H}_{b}^{p+1}(M)$ by Lemma \ref{surjharmonic} there exists $\sigma\in\mathcal{H}_b^{p-1}(M)$ such that $L\sigma=h$ and since $p-1\geq n$ by Lemma \ref{surj} there exists $\sigma_1\in F_b^{p-2}(M)$  such that $\gamma=L\sigma_1$. Thus,
\begin{displaymath}
L(\psi-\sigma-d_b\sigma_1)=0\,\,{\rm and}\,\,d_b(\psi-\sigma-d_b\sigma_1)=\varphi.
\end{displaymath} 
Then, $a=\{\varphi\}$ is the basic zero class in $H^{p}(\mathcal{A}_b(M))$ and this finishes the proof.
\end{proof}

Now following an argument similar that in \cite{F-I-L1}, we relate the  coeffective basic de Rham cohomology with the basic de Rham cohomology of $K$--contact manifolds by means of a long exact sequence in basic cohomology.

Let us consider the following short exact sequence for any degree $p$:
\begin{equation}
\label{X2}
0\longrightarrow\ker L|_{F_b^p(M)}=\mathcal{A}_b^p(M)\stackrel{i_b}{\longrightarrow}F_b^p(M)\stackrel{L}{\longrightarrow}{\rm Im}_b^{p+2}L\longrightarrow0.
\end{equation}
Since $L$ commutes with $d_b$, the sequence \eqref{X2} becomes a short exact sequence of basic differential complexes:
\begin{equation}
\label{X3}
0\longrightarrow\left(\ker L|_{F_b^p(M)},d_b\right)=\left(\mathcal{A}_b^p(M),d_b\right)\stackrel{i_b}{\longrightarrow}\left(F_b^p(M),d_b\right)\stackrel{L}{\longrightarrow}\left({\rm Im}_b^{p+2}L,d_b\right)\longrightarrow0.
\end{equation}
Therefore, we have the associated long exact sequence in cohomology \cite{Va}:
\begin{equation}
\label{X4}
\ldots\longrightarrow H^p(\mathcal{A}_b(M))\stackrel{H(i_b)}{\longrightarrow}H^p_b(M)\stackrel{H(L)}{\longrightarrow}H^{p+2}({\rm Im }_bL)\stackrel{\delta^b_{p+2}}{\longrightarrow}H^{p+1}(\mathcal{A}_b(M))\longrightarrow\ldots,
\end{equation}
where $H(i_b)$ and $H(L)$ are the induced homomorphisms in basic cohomology by $i_b$ and $L$, respectively, and $\delta^b_{p+2}$ is the connecting homomorphism defined in the following way: for $[\varphi]\in H^{p+2}({\rm Im}_bL)$, then $\delta^b_{p+2}[\varphi]=[d_b\psi]$, for $\psi\in F_b^p(M)$ such that $L\psi=\varphi$.

From Lemma \ref{surj} it results that ${\rm Im}_b^{p+2}L=F_b^{p+2}(M)$, for $p\geq n-1$. As a consequence, we have
\begin{displaymath}
H^{p+2}({\rm Im}_bL)=H^{p+2}_b(M), \,\,\forall p\geq n.
\end{displaymath}
Furthermore, the long exact sequence in basic cohomology \eqref{X4} may be expressed as
\begin{equation}
\label{X5}
\ldots\longrightarrow H^p(\mathcal{A}_b(M))\stackrel{H(i_b)}{\longrightarrow}H^p_b(M)\stackrel{H(L)}{\longrightarrow}H^{p+2}_b(M)\stackrel{\delta^b_{p+2}}{\longrightarrow}H^{p+1}(\mathcal{A}_b(M))\longrightarrow\ldots
\end{equation}
for $p\geq n$. Now, we shall decompose the long exact sequence \eqref{X5} in $5$ terms exact sequences:
\begin{equation}
\label{X6}
0\rightarrow{\rm Im}\,\delta^b_{p+1}=\ker H(i_b)\stackrel{i}{\rightarrow}H^p(\mathcal{A}_b(M))\stackrel{H(i_b)}{\longrightarrow}H^p_b(M)\stackrel{H(L)}{\longrightarrow}H^{p+2}_b(M)\stackrel{\delta^b_{p+2}}{\longrightarrow}{\rm Im}\,\delta^b_{p+2}\rightarrow0.
\end{equation}
If $H_b^p(M)$ are finite dimensional (for instance if $M$ is compact) we denote by $b_b^p(M)=\dim H_b^p(M)$ the basic $p$-th Betti number of $M$, see \cite{B-G}. Since $0\leq\dim({\rm Im}\,\delta^b_p)\leq b_b^p(M)$, for $p\geq n+2$ we have the following result:
\begin{proposition}
\label{finitedim}
Let $M$ be a $K$--contact manifold of dimension $2n+1$ such that $H_b^p(M)$ are finite dimensional. Then the coeffective basic  de Rham cohomology group $H^p(\mathcal{A}_b(M))$ has finite dimension, for $p\geq n+1$.
\end{proposition}
Thus, we can define the {\em coeffective basic numbers} of $M$ by $c_b^p(M)=\dim H^p(\mathcal{A}_b(M))$, $p\geq n+1$. Notice that $c_b^p(M)=0$ for $p\leq n-1$.

From \eqref{X6}, we have
\begin{displaymath}
\dim({\rm Im}\,\delta^b_{p+1})-\dim H^p(\mathcal{A}_b(M))+\dim H_b^p(M)-\dim H_b^{p+2}(M)+\dim({\rm \delta^b_{p+2}})=0,
\end{displaymath}
for $p\geq n+1$, from which we deduce
\begin{equation}
\label{X7}
\dim({\rm Im}\delta^b_{p+1})-c_b^p(M)+b_b^p(M)-b_b^{p+2}(M)+\dim({\rm Im}\delta^b_{p+2})=0.
\end{equation}
Now, as a consequence of \eqref{X7}, we obtain that the  coeffective basic  numbers of $M$ are bounded by upper and lower limits depending on the basic Betti numbers of the $K$--contact manifold $M$.
\begin{theorem}
Let $M$ be a $K$--contact manifold of dimension $2n+1$ such that $H_b^p(M)$ are finite dimensional. Then
\begin{equation}
\label{X8}
b_b^p(M)-b_{b}^{p+2}(M)\leq c_b^p(M)\leq b_b^p(M)+b_b^{p+1}(M)
\end{equation}
for every $p\geq n+1$.
\end{theorem}
Since $b_b^{2n}(M)=1$ and $b_b^p(M)=0$ for every $p\geq 2n+1$ we obtain
\begin{corollary}
Let $M$ be a $K$--contact manifold of dimension $2n+1$. Then $c_b^{2n}(M)=1$.
\end{corollary}
We also have
\begin{theorem}
Let $M$ be a compact Sasakian manifold of dimension $2n+1$. Then
\begin{equation}
\label{X9}
c_b^p(M)=b_b^p(M)-b_{b}^{p+2}(M),\,\,\forall p\geq n+1.
\end{equation}
\end{theorem}
\begin{proof}
The proof follows in a similar manner to the proof of Theorem 5.1. from \cite{C-L-M1} or Theorem 4.1 from \cite{F-I-L} and consist in computing the connecting mapping $\delta^b_{p+2}$.

Let $a\in H_b^{p+2}(M)$. Taking into account the Hodge theory for basic forms on compact Sasakian manifolds, see \cite{B-G}, we may consider the unique harmonic representative $\varphi$ of the basic de Rham cohomology class $a$. 

Then, by Lemma \ref{surjharmonic}, there exists a harmonic basic $p$--form $\psi$ such that $L\psi=\varphi$. The theorem follows by the definition of the connecting homomorphism, $\delta^b_{p+2}\varphi=[d_b\psi]=0$.

\end{proof}

In the end of this section we give a relation between the coeffective de Rham cohomology $H^{\bullet}(\mathcal{A}(M))$ of a compact $K$--contact manifold $M$, \cite{F-I-L}, and our basic coeffective de Rham cohomology of $M$. 

Recall that if $M$ is compact, the Lie group of isometries of the metric $g$ is compact and then the closure of the subgroup $\{\exp(t\xi)\}_{t\in\mathbb{R}}$ is a compact abelian Lie group, that is it is isomorphic to a torus $\mathcal{T}$. Denoting by $F^{\bullet}_b(M)^{\mathcal{T}}$ the complex of $\mathcal{T}$--invariant forms on $M$, then according to Proposition 7.2.1 from \cite{B-G} the following sequence
\begin{equation}
\label{X10}
0\longrightarrow F_b^\bullet(M)\stackrel{\imath}{\longrightarrow}F^\bullet(M)^{\mathcal{T}}\stackrel{\imath_{\xi}}{\longrightarrow}F_b^{\bullet-1}(M)\longrightarrow0
\end{equation}
is an exact sequence of complexes which leads to the following long exact sequence in cohomology 
\begin{equation}
\label{X12}
\ldots\longrightarrow H_b^p(M)\stackrel{\imath_*}{\longrightarrow} H^p(M)\stackrel{j_p}{\longrightarrow}H_b^{p-1}(M)\stackrel{\delta_p}{\longrightarrow}H^{p+1}_b(M)\longrightarrow\ldots,
\end{equation}
where $\delta_p$ is the connecting homomorphism given by $\delta_p[\varphi]=[L\varphi]=[d\eta]\cup [\varphi]$, and $j_p$ is the composition of the map induced by $\imath_\xi$ with the isomorphism $H^p\left(F^\bullet(M)^{\mathcal{T}}\right)\cong H^p(M)$.

Taking into account that $\imath_\xi L=L\imath_\xi$ then
\begin{equation}
\label{X11}
0\longrightarrow \mathcal{A}_b^\bullet(M)\stackrel{\imath}{\longrightarrow}\mathcal{A}^\bullet(M)^{\mathcal{T}}\stackrel{\imath_{\xi}}{\longrightarrow}\mathcal{A}_b^{\bullet-1}(M)\longrightarrow0
\end{equation}
is an exact sequence of coeffective complexes, where $\mathcal{A}^\bullet(M)$ is the space of coeffective forms on $M$, that is $\varphi\in F^\bullet(M)$ such that $L\varphi=0$ and $\mathcal{A}^\bullet(M)^{\mathcal{T}}$ is the space of coeffective $\mathcal{T}$--invariant forms.

Now, if we consider the long exact sequence in cohomology \eqref{X12} for coeffective forms we obtain that the connecting homomorphism $\delta_p$ vanish for every $p$, so we get the short exact sequence
in coeffective cohomology, 
\begin{equation}
\label{X13}
0\longrightarrow H^p(\mathcal{A}(M))\stackrel{j_p}{\longrightarrow}H^{p-1}(\mathcal{A}_b(M))\longrightarrow0,
\end{equation}
for every $p\geq 1$, which say that
\begin{theorem}
\label{isomorphic}
If $M$ is a compact $K$--contact manifold of dimension $2n+1$, then
\begin{equation}
\label{X14}
H^p(\mathcal{A}(M))\cong H^{p-1}(\mathcal{A}_b(M)),\,\,\forall\,p=1,\ldots,2n+1.
\end{equation}
\end{theorem}

\section{Coeffective basic Dolbeault cohomology}
\setcounter{equation}{0}
In this section we extend our study for basic forms with complex values on a Sasakian manifold $M$ obtaining a  coeffective basic  Dolbeault cohomology on $M$. In the case when $M$ is a compact Sasakian manifold, we prove a Hodge decomposition theorem for coeffective basic de Rham cohomology of $M$, relating this cohomology with basic coeffective Dolbeault cohomology of $M$. The notions are introduced in a similar manner as for K\"{a}hler manifolds, see \cite{I}.

For our purpose the complex valued forms on Sasakian manifolds play an important role. For this reason we have need to recall some notions about Dolbeault basic operators on Saskian manifolds. Notice that endomorphism $F$ determines a complex structure on the contact distribution $\mathcal{D}=\ker\eta$ and on a Sasakian manifold we have $N_F(X,Y)=0$ for any $X,Y\in\mathcal{D}$, where $N_F$ denotes the Nijenhuis tensor associated to $F$. Then the complexified of the space of basic $p$--forms admits the decomposition
\begin{equation}
\label{basdec}
F^p_b(M)\otimes_{\mathbb{R}}\mathbb{C}=\bigoplus_{r+s=p}F^{r,s}_b(M),
\end{equation}
where $F^{r,s}_b(M)$ is the space of \textit{basic forms of type} $(r, s)$, that is the basic forms
which can be nonzero only when act on $r$ vector fields from $\mathcal{D}^{1,0}$ and on s vector
fields from $\mathcal{D}^{0,1}$. Here we have considered the decomposition of the complexified contact distribution, namely $\mathcal{D}\otimes_{\mathbb{R}}\mathbb{C}=\mathcal{D}^{1,0}\oplus\mathcal{D}^{0,1}$. Then, by applying the classical method used in the case of almost complex manifolds (see for instance \cite{KN} pg. 125-126), a simple calculation proves that
\begin{displaymath}
d_bF^{r,s}_b(M)\subset F^{r+1,s}_b(M)\oplus F_b^{r,s+1}(M),
\end{displaymath}
and so the basic exterior derivative admits the decomposition $d_b=\partial_b+\overline{\partial}_b$, where
\begin{displaymath}
\partial_b:F^{r,s}_b(M)\rightarrow F^{r+1,s}_b(M)\,;\,\overline{\partial}_b:F^{r,s}_b(M)\rightarrow F^{r,s+1}_b(M).
\end{displaymath}
By $d_b^2=0$ we deduce 
\begin{equation}
\label{e383}
\partial_b^2=\overline{\partial}_b^2=\partial_b\overline{\partial}_b+\overline{\partial}_b\partial_b=0.
\end{equation} 

On the other hand, we have the decomposition $d^\ast_{b}\omega=\partial_b^*\omega+\overline{\partial}_b^*\omega$, induced by the decomposition $d_{b}=\partial_b+\overline{\partial}_b$ of the basic differential and some formulas similar to (\ref{e383}), namely
\begin{equation}\label{e384}
\partial_b^{*2}=\overline{\partial}_b^{*2}=\partial_b^*\overline{\partial}_b^*+\overline{\partial}_b^*\partial_b^*=0.	
\end{equation}
\par
Notice that the classical Hodge identities from K\"{a}hler geometry also hold on a compact Sasakian manifold, as shown in \cite{Ta75}. See also Lemma 7.2.7 from \cite{B-G} or Lemme 3.4.4 from \cite{E-K1} in a more general case of transversally K\"{a}hlerian foliations. 
If we define 
\begin{equation}
\label{baslap}
\Delta_b=d_bd_b^\ast+d_b^\ast d_b\,,\,\Delta_{\partial_b}=\partial_b\partial_b^*+\partial_b^*\partial_b\,\,,\,\,\Delta_{\overline{\partial}_b}=\overline{\partial}_b\overline{\partial}_b^*+\overline{\partial}_b^*\overline{\partial}_b,
\end{equation}
then we have
\begin{lemma}
\label{eqlap}
(\cite{B-G}) On a compact Sasakian manifold one has
\begin{displaymath}
\Delta_b=\Delta_{\partial_b}+\Delta_{\overline{\partial}_b}=2\Delta_{\partial_b}=2\Delta_{\overline{\partial}_b}.
\end{displaymath}
\end{lemma}

Also the equality $\overline{\partial}_b^2=0$ induces the differential complex $(F_b^{r,\bullet}(M),\overline{\partial}_b)$; its cohomology groups 
\begin{displaymath}
H_b^{r,s}(M)=\ker\{F_b^{r,s}(M)\stackrel{\overline{\partial}_b}{\rightarrow}F_b^{r,s+1}(M)\}\slash\, \overline{\partial}_b(F_b^{r,s-1}(M)),
\end{displaymath}
are the analogous of Dolbeault cohomology groups from K\"{a}hler geometry and are called the \textit{basic Dolbeault cohomology groups} of the Sasakian manifold $M$,  \cite{B-G}. In particular, there is a transverse Hodge theory for the operator $\overline{\partial}_b$, see  \cite{B-G, E-K2}.

Since $d\eta\in F_b^{1,1}(M)$, as in the previous subsection we consider the subspace $\mathcal{A}_b^{r,s}(M)\subset F_b^{r,s}(M)$ defined by
\begin{displaymath}
\mathcal{A}_b^{r,s}(M)=\{\varphi\in F_b^{r,s}(M)\,|\,\varphi\wedge d\eta=0\}=\ker L|_{F_b^{r,s}(M)}.
\end{displaymath}
A basic form  $\varphi\in\mathcal{A}_b^{r,s}(M)$ is said to be a \textit{coeffective (bigraduate) basic form} of
bidegree$(r,s)$.

From $\overline{\partial}_bd\eta=0$ 
the operator $L$ commutes with the operator $\overline{\partial}_b$. Therefore, we can consider the subcomplex of basic Dolbeault complex of $M$, namely $\left(\mathcal{A}_b^{r,\bullet},\overline{\partial}_b\right)$ for $0\leq r\leq n$; it is called the \textit{ coeffective basic Dolbeault complex} of $M$. The cohomology groups
of this subcomplex are called \textit{coeffective basic Dolbeault cohomology groups} of $M$ and they are denoted by $H^{r,s}(\mathcal{A}_b(M))$. 

Taking into account the decomposition \eqref{basdec}, we obtain the following version of Lemma \ref{surj}, when $L$ acts on $F_b^{r,s}(M)$:
\begin{lemma}
\label{surjD}
The operator $L:F_b^{r,s}(M)\rightarrow F_b^{r+1,s+1}(M)$ defined by $L\varphi=\varphi\wedge d\eta$ is injective for $r+s\leq n-1$ and surjective for $r+s\geq n-1$.
\end{lemma}
As a consequence of Lemma \ref{surjD}, one gets
\begin{proposition}
Let $M$ be a regular Sasakian manifold of dimension $2n+1$. Then $\mathcal{A}_b^{r,s}(M)=\{0\}$ for $r+s\leq n-1$, therefore
\begin{equation}
\label{D2}
H^{r,s}(\mathcal{A}_b(M))=\{0\}\,,\,\,{\rm for}\,\,r+s\leq n-1.
\end{equation}
\end{proposition}

Let us denote by  $[d\eta]_{D}$ the basic Dolbeault class of $d\eta$ in $H_b^{1,1}(M)$ and we consider the subspace of $H^{r,s}_{b}(M)$ given by the basic Dolbeault cohomology classes truncated by the class $[d\eta]_{D}$, namely,
\begin{equation}
\label{D3}
\widetilde{H}^{r,s}_{b}(M)=\{a\in H^{r,s}_{b}(M)\,|\,a\wedge[d\eta]_{D}=0\}.
\end{equation}

Next we define the mapping $\alpha_{r,s}:H^{r,s}(\mathcal{A}_b(M))\rightarrow\widetilde{H}^{r,s}_{b}(M)$ by
\begin{equation}
\label{D4}
\alpha_{r,s}\left(\{\varphi\}_{D}\right)=[\varphi]_{D},
\end{equation}
where $\{\varphi\}_{D}$ denotes the cohomology class of a coeffective basic form $\varphi$ in $H^{r,s}(\mathcal{A}_b(M))$  and  $[\varphi]_{D}$ denotes the cohomology class of a basic form $\varphi$ in $H^{r,s}_{b}(M)$.  This mapping permits us to give a relation between the coeffective basic Dolbeault cohomology groups of the Sasakian manifold $M$ and the subspaces of the basic Dolbeault cohomology groups given by \eqref{D3}, namely
\begin{proposition}
\label{surjalphaD}
If $M$ is a regular Sasakian manifold of dimension $2n+1$, the mapping $\alpha_{r,s}$ defined by \eqref{D4} is surjective for $r+s\geq n$.
\end{proposition}
\begin{proof}
It follows in a similar manner with the proof of Proposition 2.6 from \cite{I} (for K\"{a}hler manifolds) using the same technique as in Proposition \ref{alphadR}.
\end{proof}

In the following, we relate  the  coeffective basic Dolbeault cohomology groups  and the subspaces of the basic Dolbeault cohomology groups given by \eqref{D3} for compact Sasakian manifolds and we prove a coeffective version of the basic Hodge decomposition theorem for  coeffective basic Dolbeault cohomology.

Now by Lemma \ref{eqlap} and taking into account that $\Delta_{\overline{\partial}_b}$ preserves the bigraduation of basic forms, we have the following version of Lemma \ref{surjharmonic} when $L$ acts on the space $\mathcal{H}_b^{r,s}(M)=\ker\Delta_{\overline{\partial}_b}$ of harmonic basic forms of type $(r,s)$:
\begin{lemma}
\label{surjharmonicD}
The operator $L:\mathcal{H}^{s-1,s-1}_{b}(M)\rightarrow\mathcal{H}^{r,s}_{b}(M)$ is surjective for $r+s\geq n+1$.
\end{lemma}

\begin{theorem}
\label{thmain1D}
For a compact Sasakian manifold $M$ of  dimension $2n+1$, we have
\begin{equation}
\label{D7}
H^{r,s}(\mathcal{A}_b(M))\cong \widetilde{H}^{r,s}_{b}(M),
\end{equation}
for $r+s\neq n$.
\end{theorem}
\begin{proof}
It follows in a similar manner with the proof of Theorem 3.2 from \cite{I} (for compact K\"{a}hler manifolds), using the same technique as in Theorem \ref{main1dR}.
\end{proof}

Now, using the above result, by similar arguments as in the proof of Theorem 3.3 from \cite{I} (for K\"{a}hler manifolds) we will obtain a Hodge decomposition theorem for coeffective basic Dolbeault cohomology of compact Sasakian manifolds. 
\begin{theorem}
\label{thmain2D}
If $M$ is a compact Sasakian manifold of dimension $2n+1$ then we have
\begin{enumerate}
\item [i)] $\widetilde{H}^p(M)\cong\bigoplus_{r+s}^p\widetilde{H}_{b}^{r,s}(M)$.

\item [ii)] $H^p(\mathcal{A}_b(M))\cong\bigoplus_{r+s}^pH^{r,s}(\mathcal{A}_b(M))$ for $r+s\geq n+1$.
\end{enumerate}
\end{theorem}
\begin{proof}
Let $a\in \widetilde{H}^p_b(M)$ and $\varphi $ a representative of $a$. Without loss the generality we can assume that $\varphi$ is basic harmonic. From \eqref{basdec} we have the decomposition
\begin{displaymath}
\varphi=\varphi_{p,0}+\ldots+\varphi_{r,s}+\ldots+\varphi_{0,p},
\end{displaymath}
and taking into account that $\ker\Delta_b=\ker\Delta_{\overline{\partial}_b}$, from $\Delta_{\overline{\partial}_b}\varphi=\Delta_b\varphi=0$ and since $\Delta_{\overline{\partial}_b}$ preserves the bigraduation, we have
\begin{displaymath}
\Delta_{\overline{\partial}_b}\varphi_{p,0}=\ldots=\Delta_{\overline{\partial}_b}\varphi_{r,s}=\ldots=\Delta_{\overline{\partial}_b}\varphi_{0,p}=0.
\end{displaymath}
Moreover, since $d\eta$ is of bidegree $(1,1)$ basic form and $\varphi\wedge d\eta=0$, we have
\begin{displaymath}
\varphi_{p,0}\wedge d\eta=\ldots=\varphi_{r,s}\wedge d\eta=\ldots=\varphi_{0,p}\wedge d\eta=0.
\end{displaymath}
Taking into account the Hodge theory for basic forms on Sasakian manifolds, see \cite{B-G}, we have 
\begin{equation}
\label{D8}
\widetilde{H}_{b}^{r,s}(M)\cong\left\{\varphi\in\mathcal{H}_{b}^{r,s}(M)\,|\,\varphi\wedge d\eta\in\overline{\partial}_b\left(F_b^{r+1,s}(M)\right)\right\}\cong\left\{\varphi\in\mathcal{H}^{r,s}_{b}(M)\,|\,\varphi\wedge d\eta=0\right\}.
\end{equation}
Thus, part i) follows by \eqref{D8}.

Now, from part i), Theorem \ref{main1dR} and Theorem \ref{thmain1D}, we have
\begin{displaymath}
H^p(\mathcal{A}_b(M))\cong\widetilde{H}^p_b(M)\cong\bigoplus_{r+s=p}\widetilde{H}^{r,s}_{b}(M)\cong\bigoplus_{r+s=p}H^{r,s}(\mathcal{A}_b(M)),
\end{displaymath}
and it follows part ii).
\end{proof}

Let us denote by $c_b^{r,s}(M)=\dim H^{r,s}(\mathcal{A}_b(M))$.
\begin{corollary}
For a compact Sasakian manifold of dimension $2n+1$ we have
\begin{displaymath}
c_b^p(M)=\sum_{r+s=p}c_b^{r,s}(M),
\end{displaymath}
for $p\geq n+1$.
\end{corollary}

\begin{remark}
Using the same technique as in the previous subsection we can relate the  coeffective basic Dolbeault cohomology of Sasakian manifolds by means of a long exact sequence in basic cohomology and we can prove that
\begin{displaymath}
h^{r,s}_b(M)-h_b^{r+1,s+1}(M)\leq c_b^{r,s}(M)\leq h_b^{r,s}(M)+h_b^{r,s+1}(M),\,\,\forall\,r+s\geq n+1,
\end{displaymath}
where $h_b^{r,s}(M)=\dim H^{r,s}_b(M)$ are the basic Hodge $(r,s)$--numbers of $M$.

Also, when $M$ is compact we obtain
\begin{displaymath}
c_b^{r,s}(M)=h^{r,s}_b(M)-h_b^{r+1,s+1}(M),\,\,\forall\,r+s\geq n+1.
\end{displaymath}
\end{remark}

\section{Coeffective basic Bott-Chern cohomology}
\setcounter{equation}{0}

In this section we firstly define basic Bott-Chern and Aeppli cohomology of a Sasakian manifold $M$ and we obtain a Hodge-Bott-Chern decomposition theorem for basic forms of $M$. Next, in similar manner with the study of  coeffective basic de Rham and Dolbeault cohomology of $M$, we define and study a coeffective Bott-Chern cohomology for basic forms on $M$. 

\subsection{Hodge-Bott-Chern decomposition for basic forms}

In the first part of this subsection, we define the basic Bott-Chern and Aeppli cohomology groups of $M$. In the second part we define a basic Bott-Chern Laplacian and we obtain a Hodge-Bott-Chern type decomposition theorem for basic forms on $M$.

\begin{definition}
The differential complex
\begin{equation}
\ldots\longrightarrow F^{r-1,s-1}_b(M)\stackrel{\partial_{b}\overline{\partial}_{b}}{\longrightarrow}F^{r,s}_b(M)\stackrel{\partial_{b}\oplus\overline{\partial}_{b}}{\longrightarrow}F^{r+1,s}_b(M)\oplus F^{r,s+1}_b(M)\longrightarrow\ldots
\label{II1}
\end{equation}
is called the \textit{basic Bott-Chern complex} of $M$ and the \textit{basic Bott-Chern cohomology groups} of $M$ of bidegree $(r,s)$, are given by
\begin{displaymath}
H^{r,s}_{b,BC}(M)=\frac{\ker\{\partial_b:F^{r,s}_b(M)\rightarrow F^{r+1,s}_b(M)\}\cap\ker\{\overline{\partial}_b:F^{r,s}_b(M)\rightarrow F^{r,s+1}_b(M)\}}{{\rm Im} \{\partial_b\overline{\partial}_b:F^{r-1,s-1}_b(M)\rightarrow F^{r,s}_b(M)\}}.
\end{displaymath}
\end{definition}

Next, we consider the dual of the basic Bott-Chern cohomology groups, given by
\begin{displaymath}
H^{r,s}_{b,A}(M)=\frac{\ker\{\partial_b\overline{\partial}_b:F^{r,s}_b(M)\rightarrow F^{r+1,s+1}_b(M)\}}{{\rm Im}\{\partial_b:F^{r-1,s}_b(M)\rightarrow F^{r,s}_b(M)\}+{\rm Im}\{\overline{\partial}_b:F^{r,s-1}_b(M)\rightarrow F^{r,s}_b(M)\}}
\end{displaymath}
called the \textit{basic Aeppli cohomology groups} of bidegree $(r,s)$ of $M$.
\begin{proposition}
The exterior product induces a bilinear map
\begin{equation}
\wedge:H^{p,q}_{b,BC}(M)\times H_{b,A}^{r,s}(M)\rightarrow H^{p+r,q+s}_{b,A}(M).
\label{II7}
\end{equation}
\end{proposition}
\begin{proof}
Let $\varphi,\psi\in F^{r,s}_b(M)$. If $\varphi$ is $d_b$--closed and $\psi$ is $\partial_b\overline{\partial}_b$--closed then $\varphi\wedge\psi$ is $\partial_b\overline{\partial}_b$--closed. Also, if $\varphi$ is $d_b$--closed and $\psi$ is $d_b$--exact then $\varphi\wedge\psi$ is $d_b$--exact and if $\varphi$ is $\partial_b\overline{\partial}_b$--exact and $\psi$ is $\partial_b\overline{\partial}_b$--closed then $\varphi\wedge\psi$ is $d_b$--exact. 

For the last assertion, we have

$\varphi\wedge\psi=\partial_b\overline{\partial}_b\theta\wedge\psi=\frac{1}{2}d_b[(\overline{\partial}_b-\partial_b)\theta\wedge\psi+(-1)^{r+s}\theta\wedge(\partial_b-\overline{\partial}_b)\psi]$.
\end{proof}
In particular, 
\begin{displaymath}
H^{r,s}_{b,BC}(M)\times H^{n-r,n-s}_{b,A}(M)\rightarrow H^{n,n}_{b,A}(M)=H_b^{2n}(M)\cong\mathbb{R}.
\end{displaymath}

In the following, we define the \textit{ Bott-Chern Laplacian} for basic forms of type $(r,s)$ by
\begin{equation}
\Delta_{BC}^b=\partial_b\overline{\partial}_b(\partial_b\overline{\partial}_b)^\ast+\partial_b^\ast\partial_b+\overline{\partial}_b^\ast\overline{\partial}_b.
\label{III3}
\end{equation}
This operator is self-adjoint, i.e. $\langle\Delta^b_{BC}\varphi,\psi\rangle_b=\langle\varphi,\Delta^b_{BC}\psi\rangle_b$. Also, for a basic form  $\varphi\in F^{r,s}_b(M)$ we have
\begin{eqnarray*}
\langle\Delta^b_{BC}\varphi,\varphi\rangle_b &=& \langle \partial_b\overline{\partial}_b(\partial_b\overline{\partial}_b)^\ast\varphi+\partial_b^\ast\partial_b\varphi+\overline{\partial}_b^\ast\overline{\partial}_b\varphi, \varphi\rangle_b\\
&=& \langle(\partial_b\overline{\partial}_b)^*\varphi, (\partial_b\overline{\partial}_b)^*\varphi\rangle_b+\langle\partial_b\varphi, \partial_b\varphi\rangle_b+\langle\overline{\partial}_b\varphi, \overline{\partial}_b\varphi\rangle_b\\
&=& ||(\partial_b\overline{\partial}_b)^*\varphi||^2+||\partial_b\varphi||^2+||\overline{\partial}_b\varphi||^2
\end{eqnarray*}
where $||\varphi||^2=\langle\varphi,\varphi\rangle_b.$ Thus, we obtain
\begin{proposition}
$\Delta^b_{BC}\varphi=0$ if and only if $(\partial_b\overline{\partial}_b)^*\varphi=\partial_b\varphi=\overline{\partial}_b\varphi=0$.
\end{proposition}

We denote by $\mathcal{H}^{r,s}_{b,BC}(M)$ the space of $\Delta^b_{BC}$--harmonic basic forms of type $(r,s)$ on $M$.

Following the same ideas from \cite{T-Y1}, we now show that $H^{*,*}_{b,BC}(M)$ is finite dimensional by analyzing the space of its harmonic basic forms. Firstly, let us consider a related fourth-order differential operator which is elliptic (see \cite{E-K2} for general transversally Hermitian foliations), namely
\begin{equation}
\widetilde{\Delta}^b_{BC}=\partial_b\overline{\partial}_b\overline{\partial}_b^\ast\partial_b^\ast+\overline{\partial}_b^\ast\partial_b^\ast\partial_b\overline{\partial}_b+\overline{\partial}_b^\ast\partial_b\partial_b^\ast\overline{\partial}_b+\partial_b^\ast\overline{\partial}_b\overline{\partial}_b^\ast\partial_b+\overline{\partial}_b^\ast\overline{\partial}_b+\partial_b^\ast\partial_b.
\label{III5}
\end{equation} 
This operator has the same kernel as $\Delta^b_{BC}$. Indeed
\begin{displaymath}
0=\langle\varphi,\widetilde{\Delta}^b_{BC}\varphi\rangle_b=||\partial_b\varphi||^2+||\overline{\partial}_b\varphi||^2+||(\partial_b\overline{\partial}_b)^*\varphi||^2+||\partial_b\overline{\partial}_b\varphi||^2+||\partial_b^*\overline{\partial}_b\varphi||^2+||\overline{\partial}_b^*\partial_b\varphi||^2
\end{displaymath}
and the three additional terms clearly do not give any additional conditions and are automatically zero by requirement $\partial_b\varphi=\overline{\partial}_b\varphi=0$. Essentially, the presence of the second-order differential terms ensures that the spaces $\ker\Delta^b_{BC}$ and $\ker\widetilde{\Delta}^b_{BC}$ coincides. Using the classical Hodge identities for Sasakian manifolds, (see Lemma 7.2.7 from \cite{B-G}), in relation \eqref{III5}, we also obtain:
\begin{proposition}
\label{p2.3}
If $M$ is a compact Sasakian manifold of dimension $2n+1$, then
\begin{displaymath}
\widetilde{\Delta}^b_{BC}=\Delta_{\overline{\partial}_b}\Delta_{\overline{\partial}_b}+\partial_b^*\partial_b+\overline{\partial}_b^*\overline{\partial}_b.
\end{displaymath}
Moreover, the harmonic spaces $\mathcal{H}_{b}^{r+s}(M)\cap F_b^{r,s}(M)$, $\mathcal{H}^{r,s}_b(M)$ and $\mathcal{H}_{b,BC}^{r,s}(M)$ coincides, and also $d\eta$ is harmonic basic $(1,1)$--form with respect to every Laplacian: $\Delta_b$, $\Delta_{\overline{\partial}_b}$ and $\Delta_{BC}^b$, respectively.
\end{proposition}
We have now
\begin{theorem}
\label{thmain}
Let $M$ be a compact Sasakian mnaifold of dimension $2n+1$.Then
\begin{enumerate}
\item[{\rm (i)}] $\dim \mathcal{H}^{r,s}_{b,BC}(M)<\infty$;
\item[{\rm (ii)}] There is an orthogonal decomposition
\begin{equation}
F_b^{r,s}(M)=\mathcal{H}^{r,s}_{b,BC}(M)\oplus {\rm Im}(\partial_b\overline{\partial}_b)\oplus({\rm Im}\,\partial_b^*+{\rm Im}\,\overline{\partial}_b^*);
\label{x1}
\end{equation}
\item[{\rm (iii)}] There are the canonical isomorphisms: 
\begin{displaymath}
\mathcal{H}^{r,s}_{b,BC}(M)\cong H^{r,s}_{b,BC}(M)\cong H^{r,s}_{b}(M).
\end{displaymath}
\end{enumerate}
\end{theorem}
\begin{proof}
(i) Because only the highest order differential need to be kept for computing  the principal symbol of a Laplace operator, by the calculations of $\widetilde{\Delta}^b_{BC}$ from Proposition \ref{p2.3}, it follows that the principal symbol of $\widetilde{\Delta}^b_{BC}$ is equal to that of the square of the operator $\Delta_{\overline{\partial}_b}$, so it is positive. Thus $\widetilde{\Delta}^b_{BC}$ is  elliptic and hence its kernel,  $\mathcal{H}^{r,s}_{b,BC}(M)$, is finite dimensional.

With $\widetilde{\Delta}^b_{BC}$ elliptic, assertion (ii) then follows directly by applying elliptic theory. For (iii), using the decomposition of (ii), we have
\begin{equation}
\ker(\partial_b+\overline{\partial}_b)=\mathcal{H}^{r,s}_{b,BC}(M)\oplus{\rm Im}(\partial_b\overline{\partial}_b).
\label{x2}
\end{equation}
This must be so since for a form $\varphi\in F_b^{r,s}(M)$ given by $\varphi=\psi+\partial_b\overline{\partial}_b\theta+\partial_b^*\theta_1+\overline{\partial}_b^*\theta_2$, where $\psi\in \mathcal{H}^{r,s}_{b,BC}(M)$, we have $\partial_b\varphi=\overline{\partial}_b\varphi=0$ if and only if
\begin{eqnarray*}
0&=& \langle\theta_1,\partial_b(\partial_b^*\theta_1+\overline{\partial}_b^*\theta_2)\rangle_b+\langle\theta_2,\overline{\partial}_b(\partial_b^*\theta_1+\overline{\partial}_b^*\theta_2)\rangle_b \\
&=& \langle\partial_b^*\theta_1+\overline{\partial}_b^*\theta_2,\partial_b^*\theta_1+\overline{\partial}_b^*\theta_2\rangle_b\\
&=& ||\partial_b^*\theta_1+\overline{\partial}_b^*\theta_2||^2
\end{eqnarray*}
which imply $\partial_b^*\theta_1+\overline{\partial}_b^*\theta_2=0$, i.e. desired decomposition from ({\ref{x2}}). Thus every cohomology class of $H^{\bullet,\bullet}_{b,BC}(M)$ contains a unique harmonic representative and $\mathcal{H}^{r,s}_{b,BC}(M)$ $\cong H^{r,s}_{b,BC}(M)$, i.e. the first isomorphism of (iii). Since $\ker\widetilde{\Delta}^b_{BC}=\ker\Delta_{\overline{\partial}_b}$, the second isomorphism of (iii) it follows by $H^{r,s}_b(M)\cong \mathcal{H}^{r,s}_b(M)$ and $\mathcal{H}_b^{r,s}(M)=\mathcal{H}^{r,s}_{b,BC}(M)$.
\end{proof}
\begin{corollary}
If $M$ is a compact Sasakian manifold of dimension $2n+1$, then $H^{r,s}_{b, BC}(M)$ is finite dimensional.
\end{corollary}
Now, let us define the \textit{Aeppli Laplacian} for basic forms of type $(r,s)$ on $M$ by
\begin{equation}
\Delta^b_{A}+\partial_b\partial^*_b+\overline{\partial}_b\overline{\partial}_b^*+(\partial_b\overline{\partial}_b)^*\partial_b\overline{\partial}_b
\label{III7}
\end{equation}
which is not elliptic, but if we change it by 
\begin{equation}
\widetilde{\Delta}^b_{A}=\partial_b\partial_b^*+\overline{\partial}_b\overline{\partial}_b^*+\overline{\partial}_b^*\partial_b^*\partial_b\overline{\partial}_b+\partial_b\overline{\partial}_b\overline{\partial}_b^*\partial_b^*+\partial_b\overline{\partial}_b^*\overline{\partial}_b\partial_b^*+\overline{\partial}_b\partial_b^*\partial_b\overline{\partial}_b^*
\label{III8}
\end{equation}
this is elliptic.

Now, if we denote $\mathcal{H}^{r,s}_{b,A}(M)=\ker\widetilde{\Delta}^b_A\cap F_b^{r,s}(M)$, then by applying ellliptic theory arguments, similar to Theorem \ref{thmain}, we have
\begin{theorem}
\label{thmainA}
Let $M$ be a compact Sasakian manifold of dimension $2n+1$. Then
\begin{enumerate}
\item[{\rm (i)}] $\dim \mathcal{H}^{r,s}_{b,A}(M)<\infty$;
\item[{\rm (ii)}] There is an orthogonal decomposition
\begin{equation}
F^{r,s}_{b}(M)=\mathcal{H}^{r,s}_{b,A}(M)\oplus ({\rm Im}\partial_b+{\rm Im}\overline{\partial}_b)\oplus{\rm Im}(\overline{\partial}_b^*\partial_b^*);
\label{x4}
\end{equation}
\item[{\rm (iii)}] There is a canonical isomorphism: 
\begin{displaymath}
\mathcal{H}^{r,s}_{b,A}(M)\cong H^{r,s}_{b,A}(M).
\end{displaymath}
\end{enumerate}
\end{theorem}
\begin{corollary}
If $M$ is a compact Sasakian manifold, then $H^{r,s}_{b, A}(M)$ is finite dimensional.
\end{corollary}

Finally, let us remark that  $\star_b$ gives an isomorphism $H^{r,s}_{b,BC}(M)\approx H^{n-r,n-s}_{b,A}(M)$.

\subsection{Coeffective Bott-Chern cohomology for basic forms}

In this subsection we define and study a  coeffective Bott-Chern cohomology for basic forms on Sasakian manifolds.

Since the operator $L$ commutes with both operators $\partial_b$ and $\overline{\partial}_b$, we can consider the subcomplex of Bott-Chern complex of $M$
\begin{equation}
\label{C1}
\ldots\longrightarrow\mathcal{A}_b^{r-1,s-1}(M)\stackrel{\partial_b\overline{\partial}_b}{\longrightarrow}\mathcal{A}_b^{r,s}(M)\stackrel{\partial_b\oplus\overline{\partial}_b}{\longrightarrow}\mathcal{A}_b^{r+1,s}(M)\oplus\mathcal{A}_b^{r,s+1}(M)\longrightarrow\ldots
\end{equation}
for $1\leq r,s\leq n$; called the \textit{coeffective basic Bott-Chern complex} of $M$. The cohomology groups
of the complex \eqref{C1} are called \textit{coeffective basic Bott-Chern cohomology groups} of $M$ and they are denoted by $H^{r,s}_{BC}(\mathcal{A}_b(M))$. 

By Lemma \ref{surjD}, one gets
\begin{proposition}
Let $M$ be a regular Sasakian manifold of dimension $2n+1$. Then \begin{equation}
\label{C2}
H^{r,s}_{BC}(\mathcal{A}_b(M))=\{0\}\,,\,\,{\rm for}\,\,r+s\leq n-1.
\end{equation}
\end{proposition}

Since $\partial_bd\eta=\overline{\partial}_bd\eta=0$, we have  that $[d\eta]_{BC}\in H^{1,1}_{BC}(M)$ and we consider the subspace of $H^{r,s}_{b,BC}(M)$ given by the basic Bott-Chern cohomology classes truncated by the basic Bott-Chern class $[d\eta]_{BC}$, namely,
\begin{equation}
\label{C3}
\widetilde{H}^{r,s}_{b,BC}(M)=\{a\in H^{r,s}_{b,BC}(M)\,|\,a\wedge[d\eta]_{BC}=0\}.
\end{equation}

Next we define the mapping $\alpha_{r,s}:H^{r,s}_{BC}(\mathcal{A}_b(M))\rightarrow\widetilde{H}^{r,s}_{b,BC}(M)$ by
\begin{equation}
\label{C4}
\alpha_{r,s}\left(\{\varphi\}_{BC}\right)=[\varphi]_{BC},
\end{equation}
where $\{\varphi\}_{BC}$ denotes the cohomology class of a coeffective basic form $\varphi$ in $H^{r,s}_{BC}(\mathcal{A}_b(M))$ and $[\varphi]_{BC}$ denotes the basic cohomology class of a basic form $\varphi$ in $H^{r,s}_{b,BC}(M)$.  This mapping permits us to give a relation between the coeffective basic Bott-Chern cohomology groups of the Sasakian manifold $M$ and the subspaces of the basic Bott-Chern cohomology groups given by \eqref{C3}, just in the case of  coeffective basic de Rham and Dolbeault cohomology of $M$. In the following, our aim is to find a link between  the  coeffective basic Bott-Chern cohomology groups  and the subspaces of the basic Bott-Chern cohomology groups given by \eqref{C3} for compact Sasakian manifolds and to prove a coeffective version of the basic Hodge decomposition theorem for basic Bott-Chern cohomology.
\begin{proposition}
\label{surjalpha}
If $M$ is a regular Sasakian manifold of dimension $2n+1$, the mapping $\alpha_{r,s}$ defined by \eqref{C4} is surjective for $r+s\geq n$.
\end{proposition}
\begin{proof}
Let $a\in\widetilde{H}^{r,s}_{b,BC}(M)$, that is, $a\in H^{r,s}_{b,BC}(M)$ and $a\wedge[d\eta]_{BC}=0$ in $H^{r+1,s+1}_{b,BC}(M)$. Consider a representative $\varphi$ of $a$ and suppose that $\varphi\notin\mathcal{A}_b^{r,s}(M)$ (notice that if $\varphi\in\mathcal{A}_b^{r,s}(M)$, then $\varphi$ defines a basic cohomology class in $H^{r,s}_{BC}(\mathcal{A}_b(M))$ such that $\alpha\left(\{\varphi\}_{BC}\right)=a$).

Since $a\wedge[d\eta]_{BC}=0$, then there exists $\sigma\in F_b^{r,s}(M)$ such that $\varphi\wedge d\eta=\partial_b\overline{\partial}_b\sigma$. Then, from Lemma \ref{surjD}, there exists $\gamma\in F_b^{r-1,s-1}(M)$ such that $L\gamma=\sigma$. Thus, $L(\varphi-\partial_b\overline{\partial}_b\gamma)=0$ and $\partial_b(\varphi-\partial_b\overline{\partial}_b\gamma)=\overline{\partial}_b(\varphi-\partial_b\overline{\partial}_b\gamma)=0$. Hence, $\varphi-\partial_b\overline{\partial}_b\gamma$ defines a basic cohomology class  in $H^{r,s}_{BC}(\mathcal{A}_b(M))$ such that $\alpha_{r,s}\left(\{\varphi-\partial_b\overline{\partial}_b\gamma\}_{BC}\right)=a$.
\end{proof}

Now, taking into account the relation \eqref{C5}, the classical Hodge identities for Sasakian manifolds and Proposition \ref{p2.3}, we have
\begin{equation}
\label{C6}
\widetilde{\Delta}_{BC}^bL-L\widetilde{\Delta}_{BC}^b=-2i\partial_b\overline{\partial}_b,
\end{equation}
so, if $\varphi\in\mathcal{H}^{r,s}_{b,BC}(M)$ then $L\varphi\in\mathcal{H}^{r+1,s+1}_{b,BC}(M)$. 

\begin{theorem}
\label{thmain1BC}
For a compact Sasakian manifold $M$ of  dimension $2n+1$, we have
\begin{equation}
\label{C7}
H^{r,s}_{BC}(\mathcal{A}_b(M))\cong \widetilde{H}^{r,s}_{b,BC}(M),
\end{equation}
for $r+s\notin\{ n,n+1\}$.
\end{theorem}
\begin{proof}
Using an argument similar to that used in \cite{I} the proof it follows by two cases using the same technique as in Theorem \ref{main1dR}.

Case 1: $r+s\leq n-1$.

From \eqref{C2} we know that $H^{r,s}_{BC}(\mathcal{A}_b(M))=\{0\}$ for $r+s\leq n-1$. Moreover, from Theorem \ref{thmain} (the first isomorphism of (iii)), we have
\begin{equation}
\label{C8}
\widetilde{H}_{b,BC}^{r,s}(M)\cong\left\{\varphi\in\mathcal{H}_{b,BC}^{r,s}(M)\,|\,\varphi\wedge d\eta\in\partial_b\overline{\partial}_b\left(F_b^{r,s}(M)\right)\right\}\cong\left\{\varphi\in\mathcal{H}^{r,s}_{b,BC}(M)\,|\,\varphi\wedge d\eta=0\right\}.
\end{equation}
Thus, from Lemma \ref{surjD} we conclude that $\widetilde{H}^{r,s}_{b,BC}(M)=\{0\}$ for $r+s\leq n-1$. This finishes the proof for $r+s\leq n-1$.

Case 2: $r+s\geq n+2$.

We shall see that the mapping $\alpha_{r,s}$ given by \eqref{C4} is an isomorphism for $r+s\geq n+1$. From Proposition \ref{surjalpha}, it is suficient to show the injection.

Let $a\in H^{r,s}_{BC}(\mathcal{A}_b(M))$ such that $\alpha_{r,s}(a)=0$ in $\widetilde{H}_{b,BC}^{r,s}(M)$ and suppose that $\varphi$ is a representative  of $a$. Since $\alpha_{r,s}(a)=\alpha_{r,s}\left(\{\varphi\}_{BC}\right)=[\varphi]_{BC}=0$ in $\widetilde{H}_{b,BC}^{r,s}(M)$, there exists $\psi\in F_b^{r-1,s-1}(M)$ such that
\begin{displaymath}
\varphi=\partial_b\overline{\partial}_b\psi.
\end{displaymath}
Suppose $\psi\notin\mathcal{A}_b^{r-1,s-1}(M)$ (notice that if $\psi\in\mathcal{A}_b^{r-1,s-1}(M)$, then $a=0$ and we conclude the proof). Since $L$ commutes with $\partial_b$ and $\overline{\partial}_b$, then $\partial_b\overline{\partial}_b(L\psi)=L(\partial_b\overline{\partial}_b\psi)=L\varphi=0$; therefore $L\psi$ defines a basic Aeppli cohomology class $[L\psi]_A\in H^{r,s}_{b,A}(M)$. From the Theorem \ref{thmainA} (the isomorphism (iii)), we have
\begin{displaymath}
L\psi=\psi_1+\partial_b\gamma_1+\overline{\partial}_b\gamma_2,
\end{displaymath}
for $\psi_1\in \mathcal{H}^{r,s}_{b,A}(M)$, $\gamma_1\in F_b^{r-1,s}(M)$ and $\gamma_2\in F_b^{r,s-1}(M)$. Since $r+s\geq n+2$ and $\psi_1\in\mathcal{H}_{b,A}^{r,s}(M)=\mathcal{H}_b^{r,s}(M)$ by Lemma \ref{surjharmonicD} there exists $\psi_2\in\mathcal{H}_b^{r-1,s-1}(M)=\mathcal{H}_{b,A}^{r-1,s-1}(M)$ such that $L\psi_2=\psi_1$ and since $r+s-1\geq n+1$ by Lemma \ref{surjD} there exist $\sigma_1\in F_b^{r-2,s-1}(M)$ and $\sigma_2\in F_b^{r-1,s-2}(M)$ such that $\gamma_1=L\sigma_1$ and $\gamma_2=L\sigma_2$, respectively. Thus,
\begin{displaymath}
L(\psi-\psi_2-\partial_b\sigma_1-\overline{\partial}_b\sigma_2)=0\,\,{\rm and}\,\,\partial_b\overline{\partial}_b(\psi-\psi_2-\partial_b\sigma_1-\overline{\partial}_b\sigma_2)=\varphi.
\end{displaymath}
Then, $a=\{\varphi\}_{BC}$ is the zero basic class in $H^{r,s}_{BC}(\mathcal{A}_b(M))$ and this finishes the proof.
\end{proof}

Now, using the above result, by similar arguments as in the proof of Theorem \ref{thmain2D} we obtain a Hodge decomposition theorem for coeffective basic Bott-Chern cohomology of compact Sasakian manifolds. 

\begin{theorem}
\label{thmain2BC}
If $M$ is a compact Sasakian manifold of dimension $2n+1$ then we have
\begin{enumerate}
\item [i)] $\widetilde{H}^p(M)\cong\bigoplus_{r+s}^p\widetilde{H}_{b,BC}^{r,s}(M)$.

\item [ii)] $H^p(\mathcal{A}_b(M))\cong\bigoplus_{r+s}^pH^{r,s}_{BC}(\mathcal{A}_b(M))$ for $r+s\geq n+2$.
\end{enumerate}
\end{theorem}

Finally, let us denote by $c_{b,BC}^{r,s}(M)=\dim H^{r,s}_{BC}(\mathcal{A}_b(M))$. Then

\begin{corollary}
If $M$ is a compact Sasakian manifold of dimension $2n+1$, then
\begin{displaymath}
c_b^p(M)=\sum_{r+s=p}c_{b,BC}^{r,s}(M), \,\,{\rm for}\,\,r+s\geq n+2.
\end{displaymath}
\end{corollary}

\section*{Acknowledgement} The first author is supported by the Sectorial Operational Program Human Resources Development (SOP HRD), financed from the European Social Fund and by the Romanian Government under the Project number POSDRU/159/1.5/S/134378 .

\noindent 
Cristian Ida\\
Department of Mathematics and Computer Science\\
University Transilvania of Bra\c{s}ov\\
Address: Bra\c{s}ov 500091, Str. Iuliu Maniu 50, Rom\^{a}nia\\
email:\textit{cristian.ida@unitbv.ro}
\medskip 

\noindent 
Paul Popescu\\
Department of Applied Mathematics\\
University of Craiova\\
Address: Craiova, 200585,  Str. Al. Cuza, No. 13,  Rom\^{a}nia\\
email:\textit{paul$_{-}$p$_{-}$popescu@yahoo.com}

\end{document}